\documentclass[reqno,12pt,letterpaper]{amsart}

\usepackage[english]{babel}
\usepackage{amsfonts, amsmath, amssymb, amsthm}

\usepackage{amsthm}
\newtheorem{definition}{Definition}[section]
\newtheorem{lemma}{Lemma}[section]
\newtheorem{proposition}{Proposition}[section]
\newtheorem{theorem}{Theorem}[section]
\newtheorem{corollary}{Corollary}[section]
\theoremstyle{remark}

\numberwithin{equation}{section}

\newcommand{\Spin}{\mathrm{Spin}}
\def\f{\varphi}
\def\n{\nabla}
\def\s{\sigma}

\def\J{\mathcal{J}}
\def\V{\mathcal{V}}
\def\H{\mathcal{H}}

\def\S{\mathbb{S}}

\def\Ric{\mathrm{Ric}}

\def \Ca{\mathbb{O}}
\def \CM{\mathbb{C}}
\def \HM{\mathbb{H}}
\def \PM{\mathbb{P}}

\def\F{\mathrm{F}}
\def\Cl{\mathrm{Cl}}
\def\Sp{\mathrm{Sp}}
\def\SU{\mathrm{SU}}
\def\U{\mathrm{U}}
\def\SO{\mathrm{SO}}

\def\E{\mathrm{E}}

\def\spin{\mathfrak{spin}}

\def\Gr{\mathrm{Gr}}

\def\End{\mathrm{End}}

\title[Clifford Structures and Twistor Spaces]
{Twistor Spaces of Riemannian Manifolds with Even Clifford Structures}
\author{Gerardo Arizmendi}
\email{gerardo@cimat.mx}
\address{Centro de Investigaci\'on en Matem\'aticas, P.O. Box 402, 36000 Guanajuato, GTO,  M\'exico}
\author{Charles Hadfield}
\email{charles.hadfield@ens.fr}
\address{DMA, \'Ecole Normale Sup\'erieure, 45 rue d'Ulm,
75230 Paris cedex 05, France}

\begin{document}

\maketitle

\begin{abstract}
In this paper we introduce the  twistor space of a Riemannian manifold with an even Clifford structure. This notion generalizes the twistor space of quaternion-Hermitian manifolds and weak-$\Spin(9)$ structures. We also construct  almost complex structures on the twistor space for parallel even Clifford structures and check their integrability. Moreover, we prove that in some cases one can give K\"ahler and Nearly-K\"ahler metrics to these spaces.
\end{abstract}

\section{Introduction}

The notion of twistor space was first introduced by Roger Penrose in \cite{Penrose}. Following the ideas of Penrose, the twistor construction for a $4$-dimensional Riemannian manifold was developed in \cite{AHS}. This was
later generalized for even dimensional manifolds in \cite{Obrian}.
The twistor space $Z$  of an 
even-dimensional Riemannian manifold admits a natural almost complex structure, and  it is well known that such a twistor space is 
complex if and only if the manifold is self-dual for $\dim(M)=4$ and locally conformally flat for $\dim(M)\geq6$ \cite{AHS} and \cite{Obrian}. A converse theorem (the so called reverse Penrose construction) in dimension $4$ has been used to construct half-conformally flat Einstein manifolds.

In another generalization, the twistor space $Z$ of quaternion-K\"ahler manifolds was defined in \cite{Salamon}. This is an $\S^2$-bundle of pointwise Hermitian structures compatible with the quaternionic structure. It is well known that this bundle admits two almost complex structures $\J$ and $\tilde \J$, one of which is always integrable and the other is never integrable \cite{Eells-Salamon}. Moreover, the manifold $Z$  admits two Einstein metrics $h$ and $\tilde h$ such that $(Z,\J,h)$ is K\"ahler-Einstein \cite{Bergery,Salamon} and $(Z,\tilde\J,\tilde h)$ is nearly-K\"ahler \cite{Grantcharov}. 

In \cite{Friedrich}, the twistor space was defined in the context of weak $\Spin(9)$ structures on $16$-dimensional Riemannian manifolds, which correspond to rank 9 even Clifford structures \cite{Moroianu}. Additionally, this twistor construction was studied for $\mathbb R^{16}$, which carries a parallel flat even Clifford structure, the Cayley plane $\F_4/\Spin(9)$, which carries a parallel non-flat even Clifford structure, and  $\S^1\times \S^{15}$, which carries a non-parallel  even Clifford structure. In the first two cases, the twistor space admits a K\"ahler metric and in the last case the twistor space is a complex manifold which does not admit a K\"ahler metric.

In this paper, we generalize these constructions to even Clifford structures of arbitrary rank $r\geq3$, noting that ranks $3$ and $9$ constitute two of the aforementioned constructions. We construct a $\tilde\Gr(2,r)$-bundle of pointwise Hermitian structures. This bundle admits an almost complex structure and we prove theorems analogous to those in \cite{Bergery,Salamon} and  \cite{Grantcharov}.

The paper is organized as follows. Section \ref{sec2} is devoted to explain the notion of even Clifford structures. In section \ref{mainsec} we explain the construction of the twistor space of Riemannian manifolds with even Clifford structures and its almost complex structures and check their integrability.  The main results of the paper are Theorems~\ref{Tcomplex}  and \ref{TKahler} as well as Corollary~\ref{Cnearly}. 

{\em Acknowledgments:} The first author would like to thank Rafael Herrera for his encouragement and comments and Andrei Moroianu for useful discussions and hospitality during a visit to Universit\'e de Versailles-St Quentin. The second author would also like to thank Andrei Moroianu for valued guidance of his Masters during which part of this work was completed. The first author was partially supported by CONACyT scholarship, the second author by a PGSM International scholarship.

\section{Even Clifford structures}\label{sec2}

The definition of a rank $r$ even Clifford structure (or $\Cl_r^0$ structures) on a Riemannian manifold was given in \cite{Moroianu}:

\begin{definition} A rank $r$ even Clifford structure on a Riemannian manifold $(M,g)$ is an oriented rank $r$ Euclidean bundle $E$ over $M$ together with a non-vanishing algebra bundle morphism, called a {\em Clifford morphism}, $\f:\Cl^0(E)\rightarrow \End(TM)$ which maps $\Lambda^2 E$ into the bundle of skew-symmetric endomorphisms $\End^-(TM)$.
\end{definition}

This definition contains almost Hermitian structures and quaternion-Hermitian structures as particular cases, for rank $2$ and $3$ respectively. An even Clifford structure $(M,g,E)$ is called {\em
  parallel}, if there exists a metric connection $\n^E$ on $E$ such
that $\f$ is connection preserving, i.e.
\[
\f(\n^E_X\s)=\n^g_X\f(\s)
\]
for every tangent vector $X\in TM$ and section $\s$ of $\Cl^0(E)$. For instance, a manifold with a parallel $\Cl_2^0$ structure is actually a K\"ahler manifold, while a manifold with a parallel $\Cl_3^0$ structure is a quaternion-K\"ahler manifold. Manifolds with $\Spin(7)$ holonomy correspond to $8$-dimensional manifolds with $\Cl_7^0$ parallel structures. Thus, one can hope this definition gives a more general framework in which to study these geometries.

\begin{definition} \label{defflat} A parallel even Clifford structure $(M,E,\n^E)$ is called {\em
    flat} if the connection $\n^E$ is flat.
\end{definition}

The list of complete simply connected Riemannian manifolds $M$ carrying a
parallel rank $r$ even Clifford structure was found in \cite{Moroianu} and is given in the tables below.
\begin{center}
\begin{tabular}{|r|l|c|}\hline
$r$    &  $M$   &  dimension of $M$ \\
 \hline\hline
2            &   K\"ahler  &  $2m,\ m\ge 1$      \\
\hline
3 and 4           &  hyper-K\"ahler  &  $4q,\ q\ge 1$      \\
\hline
4            &  reducible hyper-K\"ahler &
$4(q^++q^-),$   $ q^+\ge 1$, $q^-\ge 1$   \\
\hline
arbitrary          & $\Cl^0_{r}$ representation space    &
multiple of $N_0(r)$     \\

\hline
\end{tabular}
\vskip .2cm
 Table 1. Manifolds with a flat even Clifford
  structure
\end{center}
\vskip .6cm

\begin{center}
\begin{tabular}{|r|l|c|}\hline
$r$ &  $M$   &  dimension of $M$ \\
 \hline\hline
2           & K\"ahler  &  $2m,\ m\ge 1$      \\
\hline
3           &  quaternion-K\"ahler 
(QK)  &  $4q,\ q\ge 1$      \\
\hline
4         &  product
of two QK manifolds& 
$4(q^++q^-)$      \\
\hline\hline
5         &  QK    &  8      \\
\hline
6           &   K\"ahler  &  8      \\
\hline
7          &  $\Spin(7)$ holonomy   &  8      \\
\hline
8         &  Riemannian  &  8      \\
\hline\hline
5   & $\Sp(k+2)/\Sp(k)\times\Sp(2)$ & $8k,\ k\ge 2$ \\
\hline
6 & $\SU(k+4)/{\rm S}(\U(k)\times\U(4))$ & $8k,\ k\ge 2$ \\
\hline
8   & $\SO(k+8)/\SO(k)\times\SO(8)$ & $8k,\ k\ge 2$ \\
\hline\hline
9      & $\hskip1.15cm\Ca \PM^2=\F_4/\Spin(9) $  &  16 \\
\hline
10     & $(\CM\otimes \Ca) \PM^2=\E_6/\Spin(10)\cdot\U(1)$   &  32 \\
\hline
12     & $(\HM\otimes \Ca) \PM^2=\E_7/\Spin(12)\cdot\SU(2)$   &  64 \\
\hline
16     & $(\Ca\otimes \Ca) \PM^2=\E_8/\Spin^+(16)$   &  128 \\
\hline
\end{tabular}
\vskip .2cm
Table 2. Manifolds with a parallel non-flat even Clifford
  structure
\end{center}
 For the sake of simplicity, in Table 2 the non-compact duals of the compact symmetric spaces have been omitted. $N_0(r)$ denotes the dimension of the irreducible representations of $\Cl_r^0$. 
For further details on even Clifford structures, we refer to \cite{Moroianu2,Moroianu}.

\section{The twistor space of an even Clifford structure}\label{mainsec}

Let $(M,g,E)$ be a manifold with even Clifford structure. Let $\f$ denote the Clifford map. Given $x\in M$, let $J_{ij}:=\f(e_i\cdot e_j)$ where $\{e_1,\dots,e_r\}$ is an orthonormal basis for $E_x$ and $\cdot$ denotes Clifford multiplication. For each $x$ we consider the subspace $Z_x$ of $\End(T_xM)$ where
\[
Z_x = \left\{J=\sum_{1\leq i<j\leq r}a_{ij}J_{ij} \,|\, J^2=-\textrm{Id}_{T_xM},\, a_{ij}\in \mathbb R\right\}
\]
and define the twistor space of the even Clifford structure to be the disjoint union
\[
Z = \bigsqcup_{x\in M} Z_x.
\]
We will denote by $\pi$ the projection onto $M$. This is a bundle of pointwise orthogonal complex structures. For a parallel $\Cl_3^0$ structure this coincides with the definition of the twistor space of a quaternion-K\"ahler manifold, where the fibre is homeomorphic to $\S^2$. It is not hard to see that for a $\Cl_4^0$ structure the fibre of the twistor space is homeomorphic to $\S^2\times \S^2$, which corresponds to the isomorphism between $\Spin(4)$ and $\Spin(3)\times \Spin(3)$. In general the fibre at each point is isomorphic to $\tilde\Gr(2,r)$, the Grassmannian of oriented $2$-planes in $\mathbb{R}^r$, as we see in the next lemma.

\begin{lemma} Let $A\in \spin(r)\subset \Cl_r^0$, then  $A^2=-1$ if and only if there exist $v_1,v_2\in \mathbb{R}^r$ orthonormal vectors such that $A=v_1\cdot v_2$.
\end{lemma}
\begin{proof}
 
 Let $A\in \spin(r)\subset \Cl_0^r$ and $\{e_1,\dots,e_r\}$ an orthornormal basis for $\mathbb{R}^r$, then under a change of basis, we can suppose that $A=\sum_{i=1}^{[\frac{r}{2}]} a_i e_{2i-1}\cdot e_{2i}$. The condition $A^2=-1$ yields the equations
\[\sum_{i=1}^{[\frac{r}{2}]}a_i^2=1, \mbox{\hspace { 1 cm}} a_ia_j=0, (i<j). \]
The solutions of these equations are the $r-$tuples $(\pm1,0,\dots,0)$, $(0,\pm1,0,\dots,0)$, $\dots$, $(0,\dots,\pm1)$. Therefore   
$A=\pm e_{2i-1}\cdot e_{2i}$ for some $1\leq i\leq [\frac{r}{2}]$.
Conversely, if $A=v_1\cdot v_2$ with $v_1$ and $v_2$ orthonormal, then $A^2=v_1\cdot v_2\cdot v_1\cdot v_2=-v_1^2v_2^2=-1$, which proves the assertion.  
\end{proof}

{\bf Remark:} Another way to prove this is by using that an element $A$ in $\Lambda^2\mathbb{ R}^r$  is decomposable if and only if $A \wedge A=0$. 

\subsection{Almost complex structures on the twistor space}

Consider $\tilde\Gr(2,r)$ as a Hermitian-symmetric space. Its complex structure can be given using Clifford multiplication. Let $z\in\tilde\Gr(2,r)$, which for a suitable frame can be written as $z=e_1\wedge e_2=e_1\cdot e_2$. The tangent space $T_{z}\tilde\Gr(2,r)$ can be identified with 
\[
\textrm{span}(e_i \cdot e_j \,|\, i\in\{1,2\} , j\in\{3,\dots, r\} )= \left\{\sum_{s=3}^r{\alpha_s e_1\cdot e_s+\beta_s e_2\cdot e_s} \,|\, \alpha_s,\beta_s\in \mathbb R \right\}.
\] 
The complex structure is then given by $\tilde{J}_{z}(v)=z\cdot v$.

The Levi-Civita connection on $M$ induces a connection on $Z$. For each $S\in Z$, the connection gives a splitting $T_SZ=V_S\oplus H_S$ where $V_S=\mathrm{ker}(\pi_*)$ is isomorphic to $T_{\pi(S)} \tilde\Gr(2,r)$, the isomorphism given by the differential of the Clifford map, and $H_S$, the horizontal subspace, is isomorphic to $T_{\pi(S)}M$. We recall the usual construction of almost complex structures on $Z$. Given $U\in V_S$ and $X\in H_S$
we define  \[\J (U+X)_S=\hat{J}(U)+\pi^{-1}_*(S\pi_*(X))\] where  $\hat {J}(U)=\f_*\tilde{J}\f_*^{-1}(U)=SU$.

 For rank $3$, this is the construction of the almost complex structure for quaternion-K\"ahler manifolds, so $(Z,\J )$ is a complex manifold, see \cite{Bergery, Salamon}. For a rank $4$ parallel even Clifford structures, the manifold is locally a product of two quaternion-K\"ahler manifolds \cite{Moroianu}. On the other hand, $\tilde \Gr(2,4)$ is isomorphic to $\S^2\times \S^2$ as K\"ahler manifolds, so in this case the twistor space is the  product of the twistor spaces of two quaternion-K\"ahler manifolds. In particular, it is a complex manifold.

From now on, we will suppose that  $(M,g)$ carries a parallel even Clifford structure of rank $r\geq5$. We treat the 8-dimensional case first. In this case, the rank should be $5,6,7$ or $8$ and the following holds: 
\begin{itemize}
\item $r= 5$: In this case the manifold is known to be quaternion-K\"ahler \cite{Moroianu}. The twistor space has fibre isomorphic to $\Sp(2)/\U(2)$ and has been considered in \cite{Burstall}. The twistor space is complex exactly when $M$ is locally symmetric.
\item $r= 6$: In this case the manifold is known to be K\"ahler \cite{Moroianu}. The twistor space has fibre $\tilde\Gr(2,6)$ and has been considered in \cite{Obrian,Burstall}. The twistor is complex exactly when the Bochner tensor of $M$ vanishes.
\item $r= 7$: In this case the manifold has $\Spin(7)$ holonomy. The twistor space has fiber $\SO(7)/\SO(5)\times \SO(2)$. According to \cite{Burstall}, the twistor space in this case is never complex.
\item $r=8$: In this case the manifold is Riemannian. The twistor fibre is isomorphic to $\SO(8)/\U(4)$ which is the usual fibre of the twistor space defined for even dimensional Riemannian manifolds. As mentioned in the introduction, the twistor space is complex if and only if $M$ is conformally flat.
\end{itemize}
Now we will assume that $n\neq8$.

\begin{lemma}\label{curvatura} Let M be a Riemannian manifold of dimension $n\neq 8$ carrying a parallel even Clifford structure of rank $r>4$, then for every  $S\in Z$ and $X,Y\in T_{\pi(S)}M$, the curvature $R$ of the Levi-Civita connection satisfies
\[[R_{SX,SY},S]-S[R_{SX,Y},S]-S[R_{X,SY},S]-[R_{X,Y},S]=0.\]
\end{lemma}

\begin{proof}
It suffices to prove the proposition for $S=J_{12}$. If the parallel even Clifford structure is flat, then by Theorem $2.9$ in \cite{Moroianu} the manifold is flat, so $R(X,Y)=0$ for all $X,Y \in T_pM$. If the parallel even Clifford structure is not flat and $n\neq8$, the proof of Proposition $2.10$ in \cite{Moroianu}, explicitly Equation~(15), implies the existence of a non-zero constant $\kappa$ such that
\begin{align}\label{curveq}
[R_{X,Y},J_{12}]&=\kappa\sum_{s>2}g(J_{s1}X,Y)J_{s2}-g(J_{s2}X,Y)J_{s1}.
\end{align}
Using the properties of the endomorphisms $J_{ij}$, Lemma 2.4~\cite{Moroianu}, specifically, $J_{ij}\circ J_{ik}=J_{jk}$ for $i,j,k$ mutually distinct, the result follows upon summing the following four calculations.
\begin{align*}
[R_{J_{12}X,J_{12}Y},J_{12}]&=\kappa\sum_{s>2}-g(J_{s1}X,Y)J_{s2}+g(J_{s2}X,Y)J_{s1} \\
-J_{12}[R_{J_{12}X,Y},J_{12}]&=\kappa\sum_{s>2}-g(J_{s2}X,Y)J_{s1}+g(J_{s1}X,Y)J_{s2} \\
-J_{12}[R_{X,J_{12}Y},J_{12}]&=\kappa\sum_{s>2}-g(J_{s2}X,Y)J_{s1}+g(J_{s1}X,Y)J_{s2} \\
-[R_{X,Y},J_{12}]&=\kappa\sum_{s>2}-g(J_{s1}X,Y)J_{s2}+g(J_{s2}X,Y)J_{s1} \qedhere
\end{align*}
\end{proof}

\begin{theorem}\label{Tcomplex}
Let M be a Riemannian manifold of dimension $n\neq 8$ carrying a parallel even Clifford structure of rank $r>4$, then the almost complex structure $\J $ on $Z$ is integrable.
\end{theorem}

\begin{proof}   
We proceed as in $14.68$ of  \cite{Besse}. For an arbitrary vector field $W$, we let $\V  (W)$ denote the vertical part of $W$ and $\H  (W)$
the horizontal part of $W$. Let $N_\J $ be the Nijenhuis tensor of $\J $. Let $U$ and $V$ be vertical vector fields and $X$ and $Y$ basic horizontal vector fields. 

Let us first check that $N_\J (U,V)=0$. Since $U$ and $V$ are vertical,  $\J (U)$ and $\J (V)$ are also vertical vector fields. Thus  $N_\J (U,V)=N_{\hat{J}}(U,V)=0$, since $\hat{J}$ is a complex structure.

Now we will check that $N_\J (X,U)=0$. From the two facts that the horizontal transport of the horizontal distribution respects $\hat{J}$, and that $[X,U]$ is vertical if $U$ is, we obtain $[X,\J U]=\J [X,U]$. This reduces the Nijenhuis tensor to $N_\J (X,U)=\J ([\J (X),U])-[\J (X),\J (U)]$. The vertical part of this vanishes by noting that both terms in  $\V ([\J (X),J(U)])=\J (\V [J(X),U])$ are tensorial in $X$. Finally, for the horizontal part of $N_\J (X,U)$ observe first that $\pi_*([J(X),U])=-U\pi_*X$ from which we obtain
\begin{align*}
\pi_*(\J[\J (X),U])&=\pi_*(\J \H [\J (X),U])\\
&=\pi_*(\pi_*^{-1}S \pi_* [\J (X),U]) \\
&=-SU\pi_*X
\end{align*}
By the same reasoning
\begin{align*}
\pi_*([\J (X),\J (U)])&=-\J (U)\pi_*X\\
&=-SU\pi_*X
\end{align*}
and so $N_\J (X,U)=0$.

Finally, we check that $N_\J (X,Y)=0$. This is done by considering the horizontal and vertical components separately. For the horizontal component, we consider $S$ in $Z$ with $\pi(S)=x$ as a section, also denoted $S$, of $Z$ about $x$ and demand that $\n S=0$ at $x$. This gives a local almost complex structure on a neighborhood of $x$ which has an associated Nijenhuis tensor $N_S$. A direct calculation gives agreement, on the neighbourhood of $x$, between the two Nijenhuis tensors considered, explicitly,
\[
\pi_*(N_\J (X,Y)_S)=N_S(\pi_*(X),\pi_*(Y)).
\]
The tensor $N_S$ is then seen to vanish at $x$ as $\n$ is torsion-free and, at $x$, $\n S$ vanishes. Studying the vertical component one recalls O'Neill's formulas for Riemannian submersions (see Chapter 9, \cite{Besse}). In particular, $\V[X,Y]_{\pi(S)}=-[R_{\pi_*X,\pi_*Y}, S]$, which implies $\V (N_\J (X,Y))=0$ precisely by Lemma~$\ref{curvatura}$.
\end{proof}

\begin{theorem}\label{TKahler}
The twistor space $(Z,\J )$ of a Riemannian manifold of dimension $n\neq8$ with a parallel even Clifford structure of rank $r>4$ and
$\Ric>0$ admits a K\"ahler metric.
\end{theorem}
\begin{proof}
 In this case the manifold $(M,g)$ is Einstein with $\Ric=\kappa(n/4+2r-4)$ (Proposition~2.10 \cite{Moroianu}). Using the condition that $\Ric >0$, we choose a metric  $h$ on $Z$ such that $\pi$ is a Riemannian submersion with totally geodesic fibres isometric to $\tilde \Gr(2,r)$ with K\"ahler metric and $\Ric=2r\kappa$, so that the collection $\{J_{ij}\}$ forms a mutually orthogonal frame and $\| J_{ij}\|^2 = 1/\kappa$. Let $U$ and $V$ be vertical vector fields and $X$ and $Y$ basic horizontal vector fields. The theorem follows a similar argument to that given in 14.81 of \cite{Besse}. We consider seperately the four cases coming from $(\n_E\J) F$ where $E,F$ may be horizontal or vertical.

First we show $\n_U\J=0$. Restricting to its action on a vertical field, we immediately get $(\nabla_U\J )V=0$ as the fibre is K\"ahler and totally geodesic. In order to prove $(\nabla_U\J )X=0$, it suffices to consider only the horizontal component (again as the fibres are totally geodesic). By appropriately choosing a local orthonormal frame for $E$, we may assume that $S=J_{12}$ and $U=\lambda J_{s1}$ with $s>2$. The Koszul formula and the relationship between the vertical component of the Lie bracket and the curvature mentioned in the previous proof give, at $S$,  
\begin{align*}
2h(\nabla_UX,Y)&=-h([X,Y],U)\\
&=\lambda h([R_{\pi_* X, \pi_* Y},J_{12}], J_{s1}).
\end{align*}
Recalling Equation~\ref{curveq} we deduce $h(\nabla_UX,Y)=-\frac{1}{2}\lambda g(J_{s2} \pi_*X, \pi_*Y)$. Using this result, $\pi_*(\nabla_UX)_S=-\frac{1}{2}\lambda J_{s2}\pi_*X$, we obtain
\begin{align*}
\pi_*(\J \nabla_UX)_S&=-\frac{1}{2} J_{12}\lambda J_{s2}\pi_*X\\
&=\frac{1}{2}U\pi_*X.
\end{align*}
Similarly, one proves that
\[
\pi_*(\nabla_U \J X)=\frac{1}{2}U\pi_*X,
\]
from which we conclude $\pi_*((\nabla_U\J )X)=0$.

Second, we show $\n_X\J=0$. Recall O'Neill's $A$ tensor
\begin{align*}
A_EF=\V \nabla_{\H E}\H F + \H \nabla_{\H E}\V F
\end{align*}
where $E$ and $F$ are arbitrary vectors. We show, as an initial calculation, that $A_X(\J  Y)=\J (A_XY)$ and $\J (A_XU)=A_X(\J U)$. In our situation we note the following decomposition into horizontal and vertical components
\begin{align*}
\n_X U &=\V \n_XU + A_XU \\
\n_X Y &=A_XY + \H \n_XY.
\end{align*}
By Proposition~9.24 in \cite{Besse}, we have $A_XY=\frac{1}{2}\V [X,Y]$ so at $S=J_{12}$, we get $A_XY=-\frac{1}{2}[R_{\pi_*X,\pi_*Y},J_{12}]$ and the claim that $A_X(\J Y)=\J(A_XY)$ is equivalent to
\[
[R_{X',J_{12}Y'},J_{12}]=J_{12}[R_{X',Y'},J_{12}]
\]
 where,for the sake of notation, we have denoted  $X'=\pi_* X$ and $Y'=\pi_* Y$. Equation~\ref{curveq} gives the result as
\begin{align*}
J_{12}[R_{X',Y'},J_{12}]&=J_{12}(\kappa\sum_{s>2}g(J_{s1}X',Y')J_{s2}-g(J_{s2}X',Y')J_{s1})\\
&=\kappa\sum_{s>2}-g(J_{s1}X',Y')J_{s1}-g(J_{s2}X',Y')J_{s2}
\end{align*}
and similarly
\begin{align*}
[R_{X',J_{12}Y'},J_{12}]&=\kappa\sum_{s>2}g(J_{s1}X',J_{12}Y')J_{s2}-g(J_{s2}X',J_{12}Y')J_{s1}\\
&= \kappa\sum_{s>2}-g(J_{s2}X',Y')J_{s2}-g(J_{s1}X', Y')J_{s1}.
\end{align*}
For the second claim, we use the skew symmetry of $A$, $h(A_XY,U)=-h(Y,A_XU)$ to obtain
\begin{align*}
h(\J A_XU,Y) &= h( U, A_X \J Y) \\
&= h(U,\J A_XY)\\
&=h(A_X\J U,Y).
\end{align*}
Therefore $\J(A_XU)=A_X(\J U)$.

We apply this result to $(\n_X\J)U$ where
\begin{align*}
(\n_X\J)U
	& = \n_X(\J U) - \J \n_XU \\
	& = \V \n_X(\J U) + A_X\J U - \J\V\n_XU  - \J A_XU  \\
	& = \V \n_X(\J U) - \J\V\n_XU 
\end{align*}
Taking the inner product of each term with $V$ and studying the respective Koszul formulas gives the result that $(\n_X \J)U=0$. By a similar calculation for $(\n_X\J)Y$,
\[
(\n_X\J)Y = \H \n_X(\J Y) - \J\H\n_XY.
\]
As this is horizontal we may use a similar idea to that presented in the preceding proof. Specifically, we may consider $S\in Z$ with $x=\pi(S)$ as a section over a neighborhood of $x$ with $\n S=0$ at $x$. Studying the appropriate Koszul formulas one concludes that, at $x$,
\[
(\n_X\J)Y=(\n_{\pi_*X}S)\pi_*Y.
\]
The result now follows since, at $x$,
\begin{align*}
(\n_X\J)Y &= \pi_*^{-1} ( \n_{\pi_* X} (S\pi_* Y) - S \n_{\pi_* X} \pi_* Y) \\
&= \pi_*^{-1} ((\n_{\pi_*X} S)\pi_*Y) = 0.\qedhere
\end{align*}
\end{proof}

The following tables summarize our findings.
\begin{center}
\begin{tabular}{|c|c|c|c|c|}
\hline
$r$ & $M$ & $\dim(M)$ & fibre of $Z$ & type of $Z$\\
\hline
\hline
$r$ & $\Cl_r^0$ representation & $N_0(r)n$  & $\tilde\Gr(2,r)$ & complex, K\"ahler
\\
\hline
3& QK manifold & $4n$ & $\S^2$ & complex, K\"ahler if $\Ric>0$
\\
\hline
4& $M_1\times M_2$, $M_i$ QK & $4(n_1+n_2)$ & $\S^2\times \S^2$ & complex, K\"ahler if $\Ric(M_i)>0$ \\
\hline
5& QK &8&$\Sp(2)/\U(2)$ & complex if locally symmetric\\
\hline
6& K\"ahler&8&$\U(4)/\U(2)\times \U(2)$& complex if Bochner tensor $\equiv 0$ \\
\hline
7& $\Spin(7)$ holonomy&8&$\tilde\Gr(2,7)$& not complex\\
\hline
8& Riemannian&8&$\SO(8)/\U(4)$& complex if Weyl tensor $\equiv 0$\\
\hline
\end{tabular}
\vskip .2cm
Table 3. Twistor spaces for low rank and dimension 8
\end{center}
\vskip .6cm

\begin{center}
\begin{tabular}{|c|c|c|c|c|}
\hline
$r$ &$M$&$\dim(M)$&fibre of $Z$&type of $Z$\\
\hline
\hline
5&$\Sp(k+2)/(\Sp(k)\times\Sp(2))$&8$k$, $k\geq 2$&$\Sp(2)/\U(2)$ &K\"ahler\\
\hline
6&$\SU(k+4)/\mathrm{S}(\U(k)\times\U(4))$&$8k$, $k\geq 2$&$\U(4)/\U(2)\times \U(2)$& K\"ahler\\
\hline
8&$\SO(k+8)/(\SO(k)\times\SO(8))$&$8k$, $k\geq 2$&$\SO(8)/\U(4)$& K\"ahler\\
\hline
9&$\F_4/\Spin(9)$&16&$\tilde\Gr(2,9)$&K\"ahler\\
\hline
10&$\E_6/(\Spin(10)\cdot\U(1))$&32&$\tilde\Gr(2,10)$&K\"ahler\\
\hline
12&$\E_7/(\Spin(12)\cdot\SU(2))$&64&$\tilde\Gr(2,12)$&K\"ahler\\
\hline
16&$\E_8/\Spin^+(16)$&128&$\tilde\Gr(2,16)$&K\"ahler\\
\hline
\end{tabular}
\vskip .2cm
Table 4. Twistor spaces for higher rank
\end{center}

For the non-compact dual spaces of these symmetric spaces, the twistor space is only complex as the negative curvature obstructs the construction of an appropriate metric on the fibres.

Finally, we use the following observation of Nagy \cite{Nagy1} to construct nearly K\"ahler metrics on the twistor space. Consider a Riemannian submersion with totally geodesic fibres
\[F \rightarrow (Z, h) \rightarrow M\]
and let $T Z= V \oplus H$ be the corresponding splitting of $T Z$. Suppose that
$Z$ admits a complex structure $\J$ compatible with $h$ and preserving $V$ and $H$ such
that $(Z, \J, h )$ is a K\"ahler manifold. Consider now the Riemannian metric on $Z$
defined by
\[
\tilde h(X, Y ) = \frac{1}{2}
 h(X, Y )\mbox{ for }X, Y \in V,
 \]
\[
 \tilde h(X, Y )  =  h(X, Y )\mbox{ for }X, Y \in H.
 \]
 The metric $\tilde h$ admits a compatible almost complex structure $\tilde \J$ given by $\tilde \J_{|V }= -\J$
and $\tilde\J_{|H} = \J$. The next proposition is proved in \cite{Nagy1}.
\begin{proposition}\cite{Nagy1} The manifold $(Z,\tilde \J, \tilde h)$ is nearly K\"ahler.
\end{proposition}
\qed
\begin{corollary}\label{Cnearly} 
The twistor space $Z$ of a Riemannian manifold with a parallel even Clifford structure of rank $r\geq3$ and
$\Ric>0$, admits an almost complex structure $\tilde\J $ and a metric $\tilde h$ such that $(Z,\tilde \J , \tilde h)$ is nearly K\"ahler.

\end{corollary}
\qed

In our case, using the definition of the almost complex structure, one can easily check that this almost complex structure is never integrable.

We conclude  by pointing out that even though a classification of parallel even Clifford strucures was given in \cite{Moroianu}, and one can try to deal with each of these cases seperately, our approach does not rely on this classification (except for dimension $8$ in which the curvature condition is not automatically satisfied). Furthermore, the constructions above can be studied in a more general context. One could check integrability conditions of these twistor spaces for manifolds with non parallel even Clifford structures, as in \cite{Friedrich}. In fact, in order for the twistor space to be complex, Lemma \ref{curvatura} should be satisfied for every $S$ in the twistor space. One nice example is given by $\S^1\times \S^{15}$, which admits a non parallel $\Cl_9^0$ structure but its twistor space is a complex manifold which cannot be  K\"ahler since its first Betti number is odd.

\end{document}